\documentclass[11pt,reqno]{amsart}

\usepackage{amsfonts,amsthm,amsmath,amssymb,amscd}
\usepackage{mathtools}
\usepackage{esint}
\usepackage{graphics}
\usepackage{indentfirst}
\usepackage{cite}
\usepackage{latexsym}
\usepackage[dvips]{epsfig}
\usepackage{mathrsfs}
\usepackage{color}
\date{}
\allowdisplaybreaks \allowdisplaybreaks[2]

\numberwithin{equation}{section}

\usepackage{graphics}
\usepackage{epsfig}
\newtheorem{theorem}{Theorem}[section]
\newtheorem{lemma}[theorem]{Lemma}
\newtheorem{proposition}{Proposition}
\newtheorem{corollary}[theorem]{Corollary}
\newtheorem{remark}{Remark}

\theoremstyle{definition}

\numberwithin{equation}{section}

\begin{document}

\title{Biconservative hypersurfaces with constant scalar curvature in space forms}

\author{Yu Fu}

\address[Fu]{School of Data Science and Artificial Intelligence, Dongbei University of Finance and Economics,
Dalian 116025, P. R. China}\email{yufu@dufe.edu.cn}

\author{Min-Chun Hong}
\address[Hong]{Department of Mathematics, The University of Queensland, Brisbane,
QLD 4072, Australia} \email{hong@maths.uq.edu.au}

\author{Dan Yang}
\address[Yang]{School of Mathematics, Liaoning University, Shenyang, 110036, China}
\email{dlutyangdan@126.com}

\author{Xin Zhan}
\address[Zhan]{School of Mathematics and Statistics, Changshu Institute of Technology,
SuZhou 215500, P. R. China} \email{zhanxin\_math@163.com}





\begin{abstract}
Biconservative hypersurfaces are hypersurfaces which have
conservative stress-energy tensor with respect to the bienergy,
containing all minimal and constant mean curvature hypersurfaces.
The purpose of this paper is to study biconservative hypersurfaces
 $M^n$ with constant scalar curvature in a space form $N^{n+1}(c)$. We prove that every
biconservative hypersurface with constant scalar curvature in
$N^4(c)$ has constant mean curvature. Moreover, we prove that any
biconservative hypersurface with constant scalar curvature in
$N^5(c)$ is ether an open part of a certain rotational hypersurface
or a constant mean curvature hypersurface. These solve an open
problem proposed recently by D. Fetcu and C. Oniciuc for $n\leq4$.
\end{abstract}

\maketitle \markboth{Fu et al.} {Biconservative hypersurfaces with
constant scalar curvature in forms}

\section{Introduction}
In 1983, Eells and Lemaire \cite{Eells} introduced the biharmonic
map to classify maps between two Riemannian manifolds $(M^n,g)$ and
$(N^m,\widetilde{g})$. A biharmonic map
$\phi:(M^n,g)\longrightarrow(N^m,\widetilde{g})$ is defined as a
critical point of the bienergy functional
$E_2(\phi)=\frac{1}{2}\int_M|\tau(\phi)|^2dv_g $ where $\tau(\phi)$
is the tension field associated to $\phi$.  By computing the first
variational formulae, a biharmonic map $\phi$ is characterized by
the vanishing of the associated bitension field:
\begin{align}\label{Euler-Lagrange}
\tau_2(\phi):=-\Delta\tau(\phi)-{\rm trace}\,
R^{N}(d\phi,\tau(\phi))d\phi=0.
\end{align}
A submanifold $M^n$ of $N^m$ is called a \textit{biharmonic}  if the
isometric immersion $\phi$ that defines the biharmonic submanifold
is a biharmonic map. The study of biharmonic submanifolds has
attracted great attentions since then in geometry. Many important
research results have been carried out to investigate the existence
and classification problems of biharmonic submanifolds in some model
spaces. More details and an overview about the historic development
concerning biharmonic submanifolds can be found in
\cite{Ou-Chen-book2020, fetcu1} and the references therein.

During the study of  biharmonic submanifolds, Caddeo et al.
\cite{CMOP2014} introduced the concept of {\em biconservative
immersions} from the principle of a stress-energy tensor for the
bienergy. An isometric immersion
$\phi:(M^n,g)\longrightarrow(N^m,\widetilde{g})$ is said to be
\textit{biconservative} if its associated divergence of the
stress-bienergy tensor $S_2$  is zero.
 Note that
the biconservative immersion $\phi$ is given by the vanishing of the
tangent part of the bitension field (c.f.
\cite{CMOP2014,loubeau2008}).   Hence, biharmonic submanifolds are
automatically biconservative submanifolds.

In  \cite{CMOP2014}, Caddeo et al.  made a first contribution on the
classification problem of biconservative surfaces in 3-dimensional
Riemannian space forms $N^3(c)$. Soon after, the first author
\cite{Fu2015Lorentz} found a new biconservative surface in the
hyperbolic 3-space $\mathbb H^3$ and provided a complete explicit
classification of biconservative surfaces in 3-dimensional
Riemannian space forms $N^3(c)$. Later, Montaldo, Oniciuc and Ratto
\cite{montaldo3} studied biconservative surfaces in an arbitrary
Riemannian manifold; in particular, they found a remarkable property
that the Hopf differential associated to a biconservative surface in
an arbitrary Riemannian manifold is holomorphic if and only if the
surface has constant mean curvature. The local parametric equations
of biconservative hypersurfaces in 4-dimensional space forms were
obtained in \cite{Hasanis-Vlachos1995H, Turgay-Upadhyay2019}.
Furthermore, the global and uniqueness properties of biconservative
surfaces (or hypersurfaces) have been investigated in a series of
papers \cite{fetcu2, fetcu3, montaldo, nistor2, nistor3}.

On the other hand, the study of minimal hypersurfaces of constant
scalar curvature $R$ in the unit sphere $\mathbb S^{n+1}$ has
received great attentions in geometry. In 1960's, Simons
\cite{Simons1968}, Lawson \cite{Lawson1969} and Chern-do
Carmo-Kobayashi \cite{Chern1970} pioneered the study of minimal
hypersurfaces with constant scalar curvature in the unit sphere. In
particular, S. S. Chern  proposed a famous conjecture: {\it for a
closed minimal hypersurface with constant scalar curvature $R$ in
the unit sphere $\mathbb S^{n+1}$, the set of $R$ should be
discrete}. Later,  Verstraelen et al \cite{verstraelen1986}
reformulated  a stronger version of Chern's Conjecture as following:
{\em any closed minimal hypersurface in the unit sphere $\mathbb
S^{n+1}$ with constant scalar curvature is isoparametric}. In 1993,
Peng-Terng \cite{Peng-Terng1993} and Chang \cite{Chang1993-JDG}
solved Chern's Conjecture for $n=3$. Furthermore, Chang
\cite{Chang1993-CAG} proved that any closed hypersurface in the unit
sphere $\mathbb S^{4}$ with constant mean curvature (CMC) and
constant scalar curvature is isoparametric. Cheng and Wan
\cite{Cheng-Wan1993} showed that any complete hypersurface with
constant mean curvature and constant scalar curvature in a space
form $N^{4}(c)$ is isoparametric. For $n\geq4$, Chern's conjecture
remains open. Some important progress had been made recently, see
for examples \cite{Tang-Yan2020, Tang-Wei-Yan2020} and references
therein.

Since the class of biconservative hypersurfaces is a generalization
of minimal (or CMC) hypersurfaces, it is very interesting to
investigate whether Chern's conjecture holds for biconservative
 hypersurfaces  instead of using minimality (or CMC) for hypersurfaces in space
forms.   In their recent survey \cite{fetcu1}, Fetcu and Oniciuc
proposed a challenging problem on biconservative hypersurfaces as following:\\

\noindent{\bf Fetcu-Oniciuc's Problem}:~~Classify all biconservative
hypersurfaces with
constant scalar curvature in a space form $N^{n+1}(c)$.\\

For a non-CMC biconservative surface in $N^3(c)$, a relation between
the Gauss curvature $K$ and mean curvature $H$ holds that
$K=-3H^2+c$ (c.f.\cite{CMOP2014},\cite{montaldo}). Assuming that the
Gauss curvature is constant, it implies that the mean curvature $H$
is also constant. Therefore, any biconservative surface with
constant Gauss curvature (or constant scalar curvature) in $N^3(c)$
is CMC. Hence the problem of the case of $n=2$ was known.

In this paper, we firstly solve Fetcu-Oniciuc's problem for $n=3$:
\begin{theorem}\label{Theorem-1}
Any biconservative hypersurface with constant scalar curvature in
$N^4(c)$ has constant mean curvature.
\end{theorem}
As applications of Tehorem \ref{Theorem-1}, we use Chang's result
\cite{Chang1993-CAG} to obtain a global result for biconservative
hypersurfaces:
\begin{corollary}
Any closed biconservative hypersurface with constant scalar
curvature in sphere $\mathbb S^4$ is isoparametric.
\end{corollary}

Even more, using Cheng and Wan's result \cite{Cheng-Wan1993}, we
have
\begin{corollary}
Any complete biconservative hypersurface with constant scalar
curvature in a space form $N^4(c)$ is isoparametric.
\end{corollary}

Secondly,  we investigate  Fetcu-Oniciuc's problem for $n=4$.
Interestingly, we found a different class of non-CMC biconservative
hypersurfaces with constant scalar curvature in $N^5(c)$. Precisely,
we prove
\begin{theorem} \label{Theorem-2}
Any biconservative hypersurface with constant scalar curvature in
$N^5(c)$ is either CMC or contained in a certain non-CMC rotational
hypersurface, where the rotational hypersurface has two distinct
principal curvatures with $-\lambda_1=\lambda_2=\lambda_3=\lambda_4$
and  the scalar curvature $R$ of this rotational hypersurface is
$12c$.
\end{theorem}
The main idea to prove Theorem \ref{Theorem-2} (also Theorems
\ref{Theorem-1}) is based on the approach developed in
\cite{fu-hong-zhan2020} for settling Chen's biharmonic conjecture.
By converting the differential equations related to principal
curvature functions into a system of algebraic differential
equations, the solution can be determined completely. Furthermore,
making use of the elimination method  we get a polynomial function
equation and derive a contradiction. Combining these with using a
result of do Carmo and Dajczer \cite{doCarmo-Dajczer1983}, we prove
Theorem \ref{Theorem-2}. The current approach seems to be quite
useful, see also Guan-Li-Vrancken's work \cite{guanli2021}.
\begin{remark}\label{remark1}
For a non-CMC rotational hypersurface in $N^5(c)$, if its principal
curvatures satisfying a linear relation
$$-\lambda_1=\lambda_2=\lambda_3=\lambda_4,$$
then the rotational
hypersurface must be biconservative (see details in Section 5).
\end{remark}
\begin{remark}
Li and Wei showed in \cite{Liwei2007} that there are no compact
embedded rotational hypersurfaces with constant scalar curvature
$n(n-1)$ of $\mathbb S^{n+1}$ other than the Riemannian product
$S^{n-1}\Big(\sqrt{\frac{n-2}{n}}\Big)\times
S^1\Big(\sqrt{\frac{2}{n}}\Big)$ and round geodesic spheres. Then we
remark that for $c=1$, the rotational hypersurface in Theorem
\ref{Theorem-2} is not compact embedded.
\end{remark}

The paper is organized as follows. In Section 2, we recall some
background on the theory of hypersurfaces and derive a useful lemma
(Lemma 2.2) for biconservative hypersurfaces with constant scalar
curvature. In Section 3, we study the case of biconservative
hypersurfaces in 4-dimensional space forms. We establish three key
lemmas (Lemmas 3.1-3.3)  to prove
 Theorem \ref{Theorem-1}. In Section 4, we consider the case
of biconservative hypersurfaces in a 5-dimensional space form and
give a proof of  Theorem \ref{Theorem-2}. In Section 5, We give
details to verify Remark \ref{remark1}.

\section{Preliminaries}

Let $N^{n+1}(c)$ be an $(n+1)$-dimensional Riemannian space form
with constant sectional curvature $c$. For an isometric immersion
$\phi: M^n\rightarrow N^{n+1}(c)$, we denote by $\nabla$ the
Levi-Civita connection of $M^n$ and $\tilde\nabla$  the Levi-Civita
connection of $N^{n+1}(c)$. For any $X,Y,Z\in C(TM)$, the Gauss and
Codazzi equations are given by
\begin{eqnarray}
R(X,Y)Z=c\big(\langle Y,Z\rangle X-\langle X,Z\rangle Y\big) +
\langle AY,Z\rangle AX-\langle AX,Z\rangle AY,\label{Gauss-Equation}
\end{eqnarray}
\begin{eqnarray}
(\nabla_{X} A)Y=(\nabla_{Y} A)X.\label{Codazzi-Equation}
\end{eqnarray}
Note that here $A$ is the shape operator satisfying
$$(\nabla_{X} A)Y=\nabla_X(AY)-A(\nabla_X Y),$$
and the curvature tensor of $M^n$ is defined to be
\begin{equation}\label{curvature-tensor}
  R(X,Y)=\nabla_X \nabla_Y-\nabla_Y \nabla_X-\nabla_{[X,Y]}.
\end{equation}

Let us recall the definition of biconservative immersions
(c.f.\cite{CMOP2014, loubeau2008}). The stress-energy tensor $S_2$
of the bienergy is defined to be
\begin{align}\label{stress-energy}
  S_2(X,Y)=&\frac 12 \vert \tau(\phi)\vert^2\langle X,Y\rangle+\langle d\phi,\nabla\tau(\phi)\rangle \langle X,Y\rangle\\
  &-\langle d\phi(X),\nabla_Y\tau(\phi)\rangle-\langle d\phi(Y),\nabla_X\tau(\phi)\rangle,\nonumber
\end{align}
and it satisfies
\begin{equation}\label{divergence}
  div\,S_2=-\langle \tau_2(\phi),d\phi\rangle=-\tau_2(\phi)^\top.
\end{equation}
The isometric immersion
$\phi:(M^n,g)\longrightarrow(N^m,\widetilde{g})$ is said to be
\textit{biconservative} if $div\,S_2=0$.

Taking into account \eqref{Euler-Lagrange} and \eqref{divergence},
we have
\begin{proposition} \rm{(c.f.\cite{fetcu1})}
A hypersurface $\phi: M^n\rightarrow  N^{n+1}(c)$ is biconservative
if the mean curvature $H$ and the shape operator $A$ on $M^n$
satisfy
\begin{equation} \label{biharmonic condition}
A\,{\rm grad}H=-\frac{n}{2}H{\rm grad}H.
\end{equation}
\end{proposition}
According to \eqref{biharmonic condition}, it is straightforward to
see that CMC hypersurfaces are automatically biconservative in
$N^{n+1}(c)$. We call non-CMC biconservative hypesurfaces as proper
biconservative hypesurfaces in $N^{n+1}(c)$.

 From now on, we consider a proper biconservative hypersurface $M^n$ in a space form  $N^{n+1}(c)$, that is, $H$ is
not identically constant on a connected component of $M^n$, where
the number of distinct principal curvatures is constant, and all of
the principal curvature functions of the shape operator $A$ are
always smooth. If we fix $p\in M^n$, then there exists a
neighborhood $U_p$ of $p$ such that ${\rm grad}\,H  \neq 0$ on
$U_p$.

From \eqref{biharmonic condition}, we see that ${\rm grad}\,H$ is an
eigenvector of $A$ with the corresponding principal curvature
$-nH/2$. Then we can choose a suitable orthonormal frame
$\{e_1,\cdots,e_n\}$ such that $e_1={\rm grad}H/|{\rm grad}H|$ and
the shape operator $A$ has the form
\begin{eqnarray}\label{A}
A=\mathrm{diag}(\lambda_1,\cdots, \lambda_n),
\end{eqnarray}
where $\lambda_1=-nH/2, \lambda_2, \cdots, \lambda_n$ are the
principal curvatures of $M^n$. The definition of the mean curvature
means $\sum_{i=1}^n\lambda_i=nH$, and hence
\begin{eqnarray}\label{sum1}
\lambda_2+\cdots+\lambda_n=-3\lambda_1.
\end{eqnarray}
The squared length of the second fundamental form, denoted by $S$,
is defined as
\begin{eqnarray}\label{sum2}
S={\rm trace}\, A^2=\sum_{i=2}^n\lambda^2_i+\lambda^2_1.
\end{eqnarray}
By Gauss equation \eqref{Gauss-Equation}, it is easy to check that
the scalar curvature $R$ of $M^n$ is given by
\begin{equation}\label{sum3}
  R=n(n-1)c+4\lambda_1^2-S.
\end{equation}
Combining \eqref{sum2} with \eqref{sum3} gives
\begin{equation}\label{sum4}
  \sum_{i=2}^n\lambda^2_i=n(n-1)c+3\lambda_1^2-R.
\end{equation}
Since $e_1 $ is parallel to ${\rm grad}H$, it follows that
$e_1(H)\neq 0$ and $e_i(H)=0$ for $2\leq i\leq n$, and hence
\begin{eqnarray} \label{H1} e_1(\lambda_1)\neq0,\quad
e_i(\lambda_1)=0, \quad 2\leq i\leq n.\label{A2}
\end{eqnarray}
Set $ \nabla_{e_i}e_j=\sum_{k=1}^n\omega_{ij}^k e_k$ $(1\leq i,j\leq
n)$, where $\omega_{ij}^k$ is the connection coefficients of $M^n$.
By the compatibility conditions $\nabla_{e_k}\langle
e_i,e_i\rangle=0$ and $\nabla_{e_k}\langle e_i,e_j\rangle=0$ $(i\neq
j)$, we have
\begin{eqnarray}\label{compatibility}
\omega_{ki}^i=0,\quad \omega_{ki}^j+\omega_{kj}^i=0,\quad i\neq j.
\end{eqnarray}
Using Codazzi equation \eqref{Codazzi-Equation}, it yields
\begin{eqnarray}
e_i(\lambda_j)&=&(\lambda_i-\lambda_j)\omega_{ji}^j,\label{Codazzi-1}\\
(\lambda_i-\lambda_j)\omega_{ki}^j&=&(\lambda_k-\lambda_j)\omega_{ik}^j\label{Codazzi-2}
\end{eqnarray}
for distinct $i, j, k$. Using \eqref{compatibility} with
\eqref{Codazzi-1}, we then actually have
\begin{equation}\label{BICONSERVATIVE-1}
  e_1(\lambda_i)=\lambda_i\omega_{ii}^1-\lambda_1\omega_{ii}^1.
\end{equation}
Similarly to the Lemma 3.2 in \cite{fu-hong2018}, it follows from
Gauss euqation that
\begin{equation}\label{BICONSERVATIVE-2}
  e_1(\omega_{ii}^1)=(\omega_{ii}^1)^2+\lambda_1\lambda_i+c.
\end{equation}

Denote $f_k=\sum_{i=2}^n (\omega_{ii}^1)^k,~ k \in \mathbb{N}^*$.
For the sake of simplicity, we write $\lambda=\lambda_1$, $T=f_1$,
$T'=e_1(T)$, $T''=e_1e_1(T)$, $T'''=e_1e_1e_1(T)$ and
$T''''=e_1e_1e_1e_1(T)$.

By modifying an argument  in \cite{fu-hong-zhan2020}, we have
\begin{lemma}\label{f1-5}
Let $M^n$ be a proper biconservative hypersurface with constant
scalar curvature $R$ in  $N^{n+1}(c)$. Then the functions $f_1,
\cdots, f_5$ can be written as
\begin{align}\label{L6}
\begin{cases}
f_1=T,\\
f_2=T'+3\lambda^2-(n-1)c,\\
f_3=\frac{1}{2}T''-(\lambda^2+c)T+6\lambda\lambda',\\
f_4=\frac{1}{6}T'''-\frac{4}{3}(\lambda^2+c)T'-\frac{7}{6}\lambda\lambda'T+2\lambda'^2\\
  \quad\quad +\frac 72 \lambda\lambda''+\frac{n^2-10}{2}c\lambda^2-\frac 12 R\lambda^2+(n-1)c^2,\\
f_5=\frac{1}{24}T''''-\frac{5}{6}(\lambda^2+c)T''-\frac{35}{24}\lambda\lambda'T'\\
  \quad\quad-\frac{1}{24}(11\lambda\lambda''+7\lambda'^2-24\lambda^4-48c\lambda^2-24c^2)T\\
    \quad\quad+\frac{11}{8}\lambda\lambda'''+\frac{15}{8}\lambda'\lambda''
    -2\lambda^3\lambda'-\frac{134-5n^2}{12}c\lambda\lambda'-\frac{5}{12}\lambda\lambda'R.
\end{cases}
\end{align}
\end{lemma}
\begin{proof}
Put $g_1=\sum_{i=2}^n \lambda_i\omega_{ii}^1$. Taking the sum of $i$
from 2 to $n$ in \eqref{BICONSERVATIVE-2} and
\eqref{BICONSERVATIVE-1} respectively, using \eqref{sum1} we have
\begin{align}
f_2=&3\lambda^2+e_1(f_1)-(n-1)c=T'+3\lambda^2-(n-1)c,\label{Q2}\\
g_1=&\lambda T-3e_1(\lambda)=\lambda T-3\lambda'.\label{Q3}
\end{align}
By multiplying $\omega_{ii}^1$ on both sides of equation
\eqref{BICONSERVATIVE-2}, taking the sum of $i$ and using
\eqref{Q2}--\eqref{Q3} gives
\begin{align}
f_3=&\frac 12 e_1(f_2)-\lambda  g_1-cf_1\label{Q4}\\
=&\frac{1}{2}T''-(\lambda^2+c)T+6\lambda\lambda'.\nonumber
\end{align}
Put $g_2=\sum_{i=2}^n\lambda_i\big(\omega_{ii}^1\big)^2$.
Differentiating $g_1=\sum_{i=2}^n \lambda_i\omega_{ii}^1$ along
$e_1$, using \eqref{BICONSERVATIVE-1}-\eqref{BICONSERVATIVE-2}, it
follows that
\begin{align} \label{Q5}
g_2=\frac{1}{2}\big\{e_1(g_1)-\lambda \sum_{i=2}^n \lambda_i^2
+\lambda f_2+3c\lambda\big\}.
\end{align}
In view of \eqref{Q2}--\eqref{Q3} and \eqref{sum4}, the above
equation becomes
\begin{align}\label{Q6}
g_2=\lambda T'+\frac 12 \lambda'T-\frac 32
\lambda''-\big(\frac{n^2}{2}-2\big)c\lambda+\frac 12\lambda R.
\end{align}
By multiplying $(\omega_{ii}^1)^2$ on both sides of equation
\eqref{BICONSERVATIVE-2}, taking the sum of $i$ from 2 to $n$ gives
\begin{align*}
f_4=\frac{1}{3}e_1(f_3)-\lambda  g_2-cf_2.
\end{align*}
Using \eqref{Q2}, \eqref{Q4} and \eqref{Q6}, $f_4$ can be written as
\begin{align}\label{Q7}
f_4=&\frac{1}{6}T'''-\frac{4}{3}(\lambda^2+c)T'-\frac{7}{6}\lambda\lambda'T+2\lambda'^2\\
&+\frac{7}{2}\lambda\lambda''+ \frac{n^2-10}{2}c\lambda^2-\frac 12
R\lambda^2 +(n-1)c^2.\nonumber
\end{align}
Put $g_3=\sum_{i=2}^n\lambda_i^2\omega_{ii}^1$. Multiplying equation
\eqref{BICONSERVATIVE-1} by $\lambda_i$, using \eqref{sum4} and
\eqref{Q3} yields
\begin{align}\label{Q8}
    g_3=\frac 12 e_1(\sum_{i=2}^n\lambda^2_i)+\lambda g_1=\lambda^2 T.
\end{align}
Put $g_4=\sum_{i=2}^n\lambda_i\big(\omega_{ii}^1\big)^3$.
Differentiating $g_2=\sum_{i=2}^n\lambda_i\big(\omega_{ii}^1\big)^2$
with respect to $e_1$ and using
\eqref{BICONSERVATIVE-1}--\eqref{BICONSERVATIVE-2}, we have
\begin{align}\label{Q9}
g_4=\frac{1}{3}\big(e_1(g_2)+\lambda f_3-2\lambda g_3-2cg_1\big),
\end{align}
which together with \eqref{Q3}, \eqref{Q4}, \eqref{Q6} and
\eqref{Q8} yields
\begin{align}\label{Q10}
    g_4 =& \frac 1 2 \lambda T''+\frac 1 2 \lambda' T'+(\frac{1}{6}\lambda''-\lambda^3-c\lambda)T\\
        &-\frac 1 2 \lambda'''+(2\lambda^2+\frac{1}{6}R)\lambda'+\frac 16 (16-n^2)c\lambda'.\nonumber
\end{align}
By multiplying $(\omega_{ii}^1)^3$ on both sides of equation
\eqref{BICONSERVATIVE-2} and taking the sum of $i$, we have
\begin{align*}
f_5=\frac{1}{4}e_1(f_4)-\lambda  g_4-cf_3.
\end{align*}
It follows from \eqref{Q4}, \eqref{Q7} and \eqref{Q9} that
\begin{align}\label{Q11}
f_5=&\frac{1}{24}T''''-\frac{5}{6}(\lambda^2+c)T''-\frac{35}{24}\lambda\lambda'T'\\
  &-\frac{1}{24}(11\lambda\lambda''+7\lambda'^2-24\lambda^4-48c\lambda^2-24c^2)T\nonumber\\
   &+\frac{11}{8}\lambda\lambda'''+\frac{15}{8}\lambda'\lambda''
    -2\lambda^3\lambda'-\frac{134-5n^2}{12}c\lambda\lambda'-\frac{5}{12}\lambda\lambda'R. \nonumber
\end{align}
\end{proof}

\section{Biconservative hypersurfaces in $N^4(c)$}

In this section, we consider biconservative hypersurfaces with
constant scalar curvature in $N^4(c)$. In order to prove the Theorem
\ref{Theorem-1}, we need the following lemmas. Firstly, an easy
computation shows that
\begin{lemma}\label{FHZ-n3}
Let $f_k=(\omega_{22}^1)^k+ (\omega_{33}^1)^k$. Then the functions
$f_1,\cdots,f_4$ satisfy the following two algebraic equations:
\begin{align}
    &f_1^3-3f_1f_2+2f_3=0, \label{f123}\\
    &f_1^4-2f_1^2f_2-f_2^2+2f_4=0.\label{f124}
\end{align}
\end{lemma}
\begin{proof}
From the definition of $f_k$, one has
\begin{align}
  f_3= & (\omega_{22}^1+\omega_{33}^1)\big\{ (\omega_{22}^1)^2-\omega_{22}^1\omega_{33}^1+(\omega_{33}^1)^2\big \} \nonumber\\
  = & f_1\big\{f_2-\frac 12(f_1^2-f_2)\big\} \nonumber\\
  = & \frac 12 f_1(3f_2-f_1^2)  \label{ff3}
\end{align}
and
\begin{align}
  f_4= & (\omega_{22}^1)^4+(\omega_{33}^1)^4 +2 (\omega_{22}^1)^2 (\omega_{33}^1)^2-2 (\omega_{22}^1)^2 (\omega_{33}^1)^2 \nonumber\\
  = & (\omega_{22}^1+\omega_{33}^1)^2-2 (\omega_{22}^1)^2 (\omega_{33}^1)^2 \nonumber\\
  = & f_2^2  -\frac{1}{2}( f_1^2-f_2 )^2 \nonumber\\
  = & \frac 12 (f_2^2+2f_1^2f_2-f_1^4). \label{ff4}
\end{align}
Therefore, \eqref{ff3} and \eqref{ff4} give \eqref{f123} and
\eqref{f124}  respectively.
\end{proof}

By Lemma \ref{f1-5} and Lemma \ref{FHZ-n3}, we can get the following
lemma.
\begin{lemma}\label{lemma-e23T}
Let $\phi: M^3\rightarrow N^4(c)$ be an orientable proper
biconservative hypersurface with constant scalar curvature. Then we
have $e_i(T)=0$ for $i=2,3$.
\end{lemma}
\begin{proof}
Assume that $T\neq0$. Substituting $f_1,f_2,f_3,f_4$ of \eqref{L6}
into \eqref{f123} and \eqref{f124} we obtain two equations
\begin{align}
   & T''-3TT' +T^3 +(4c - 11\lambda^2) T+12\lambda\lambda'=0, \label{n31}\\
   & \frac 13 T''' -T'^2 + (\frac {4}{3}c -\frac{26}{3}\lambda^2 -2T^2 )T'+T^4 + (4c -6\lambda^2) T^2\label{n32}\\
   &  -\frac 73 \lambda\lambda'T -9\lambda^4 -\lambda^2 R + 11c\lambda^2  +7\lambda\lambda''+4\lambda'^2 =0.\nonumber
\end{align}
Differentiating \eqref{n31} with respect to $e_1$ and using
\eqref{n31}-\eqref{n32}, we can eliminate $T''', T''$ and get
\begin{align}
   5\lambda T'  -5\lambda T^2 +7\lambda' T
   +9\lambda^3 +\lambda R  -11 c\lambda   -3 \lambda''=0.\label{n33}
\end{align}
Note that $\lambda\not\equiv 0$. Differentiating \eqref{n33} with
respect to $e_1$ and using \eqref{n31}, one has
\begin{align}
   & (5 \lambda T  + 12 \lambda') T'- 5 \lambda T^3 - 5 \lambda' T^2
   + 7 \lambda'' T  +55 \lambda^3 T  \label{n34} \\
   &  - 20 \lambda T c  +  R \lambda' - 33 \lambda^2 \lambda'  - 11 c\lambda'  - 3 \lambda''' =0.\nonumber
\end{align}
Suppose that $5 \lambda T  + 12 \lambda'=0$, then $T=-\frac {12
\lambda'}{5\lambda} $, in this case the lemma follows immediately.
Assume that $5 \lambda T  + 12 \lambda' \neq 0$. Combining
\eqref{n33} with \eqref{n34} gives
\begin{align}\label{n35}
  p_1 T+p_2=0,
\end{align}
where
\begin{align*}
  &p_1= 50 \lambda\lambda''  - 84 \lambda'^2+230 \lambda^4  - 5 R \lambda^2   - 45 c \lambda^2    , \nonumber \\
  &p_2=  - 15 \lambda\lambda''' + 36\lambda'\lambda''- 273 \lambda^3 \lambda'  - 7 R\lambda\lambda' + 77 c\lambda\lambda' .
\end{align*}
\textbf{Case A:} $p_1\neq 0$. It follows that $T=-p_2/p_1$, and the lemma follows.\\
\textbf{Case B:} $p_1= 0$. It follows from \eqref{n35} that $p_2=0$.
By applying $p_1= 0$ and $p_2=0$, we can eliminate $\lambda''',
\lambda''$ and get
\begin{align}
   - 252 \lambda'^2 -60 \lambda^4 + 2485 R \lambda^2  - 12635 c \lambda^2 =0. \label{n36}
\end{align}
Differentiating \eqref{n36} along $e_1$ and using $p_1=0$ gives the
following expression:
\begin{align}
  -25980 \lambda^4 - 61495 R\lambda^2  + 321545 c\lambda^2  + 10584 \lambda'^2=0.  \label{n37}
\end{align}
Finally, combining \eqref{n36} with \eqref{n37} gives a non-trivial
quadratic equation of $\lambda$ with constant coefficients as
follows:
\begin{align}
   228 \lambda^2 - 343 R + 1673 c=0. \label{n38}
\end{align}
Hence, $\lambda$ must be a constant, which is impossible.
\end{proof}

By a similar argument as Lemma 3.4 in \cite{fu-hong-zhan2020}, we
have
\begin{lemma}\label{lambda-23}
Let $\phi: M^3 \rightarrow N^4(c)$ be an orientable proper
biconservative hypersurface. Then $e_i(\lambda_j)=0$ for $2\leq i,
j\leq 3$.
\end{lemma}

We are now in a position to prove Theorem \ref{Theorem-1}.

\vspace{0.2cm} \noindent\textbf{The proof of Theorem
\ref{Theorem-1}:}\\

 We need
only to consider two cases:

\textbf{Case 1.} Suppose that $M^3$ has two distinct principal
curvatures. Then one has $\lambda\neq\lambda_2=\lambda_3$. Due to
\eqref{sum1} we have $\lambda_2=\lambda_3=-\frac 32 \lambda$, which
together with \eqref{sum4} implies that $\lambda^2=4c-\frac 23
R=$constant. Therefore, we conclude that the mean curvature
$H=-\frac 23 \lambda$ is also a constant.

\textbf{Case 2.} Suppose that $M^3$ has three distinct principal
curvatures. Combining \eqref{sum1} with \eqref{sum4}, we may
eliminate $\lambda_3$ and get
\begin{equation}\label{Tn31}
  6\lambda^2+6\lambda\lambda_2+2\lambda_2^2-6c+R=0.
\end{equation}
Differentiating \eqref{Tn31} with respect to $e_1$ leads to
\begin{equation}\label{Tn32}
  (6\lambda+3\lambda_2)e_1(\lambda)+(3\lambda+2\lambda_2)e_1(\lambda_2)=0.
\end{equation}
Similarly, we have
\begin{equation}\label{Tn33}
  (6\lambda+3\lambda_3)e_1(\lambda)+(3\lambda+2\lambda_3)e_1(\lambda_3)=0.
\end{equation}
Taking into account \eqref{sum1} we know
$3\lambda+2\lambda_2=\lambda_2-\lambda_3\neq 0$ and
$3\lambda+2\lambda_3=\lambda_3-\lambda_2\neq 0$. Suppose that
$6\lambda+3\lambda_2=0$. Then \eqref{Tn32} implies
$e_1(\lambda_2)=0$. Hence $e_1(\lambda)=-\frac 12 e_1(\lambda_2)=0$,
which shows that $\lambda$ must be a constant. This contradicts to
our assumption. Thus, we get $2\lambda+\lambda_2\neq 0$. In the same
manner we have that $2\lambda+\lambda_3\neq 0$. Differentiating
\eqref{Tn32} along $e_1$ yields
\begin{equation}\label{Tn34}
  e_1e_1(\lambda_2)=-\frac {3(2\lambda+\lambda_2)}{3\lambda+2\lambda_2}e_1e_1(\lambda)
  -\frac{6(3\lambda^2+3\lambda\lambda_2+\lambda_2^2)}{(3\lambda+2\lambda_2)^3}\big(e_1(\lambda)\big)^2.
\end{equation}
Combining \eqref{BICONSERVATIVE-1} with \eqref{BICONSERVATIVE-2}
shows that
\begin{equation*}
  e_1\Big(\frac{e_1(\lambda_2)}{\lambda_2-\lambda}\Big)=\Big(\frac{e_1(\lambda_2)}{\lambda_2-\lambda}\Big)^2+\lambda\lambda_2+c,
\end{equation*}
which further reduces to
\begin{equation}\label{Tn35}
  \frac{e_1e_1(\lambda_2)}{\lambda_2-\lambda}+\frac {e_1(\lambda_2)e_1(\lambda)}{(\lambda_2-\lambda)^2}
  -2\Big(\frac {e_1(\lambda_2)}{\lambda_2-\lambda}\Big)^2-\lambda\lambda_2-c=0.
\end{equation}
By applying \eqref{Tn32} and \eqref{Tn34} to \eqref{Tn35}, we may
eliminate $e_1e_1(\lambda_2), e_1(\lambda_2)$, and finally obtain
\begin{align}\label{Tn36}
   & (\lambda\lambda_2+c)(\lambda-\lambda_2)^2(3\lambda+2\lambda_2)^3 \nonumber\\
   & +9(28\lambda^3+51\lambda^2\lambda_2+30\lambda\lambda_2^2+6\lambda_2^3)\big(e_1(\lambda)\big)^2\nonumber\\
   & +3(\lambda_2-\lambda)(2\lambda+\lambda_2)(3\lambda+2\lambda_2)^2e_1e_1(\lambda)=0.
\end{align}
 Similarly, one has
\begin{align}\label{Tn37}
   & (\lambda\lambda_3+c)(\lambda-\lambda_3)^2(3\lambda+2\lambda_3)^3 \nonumber\\
   & +9(28\lambda^3+51\lambda^2\lambda_3+30\lambda\lambda_3^2+6\lambda_3^3)\big(e_1(\lambda)\big)^2\nonumber\\
   & +3(\lambda_3-\lambda)(2\lambda+\lambda_3)(3\lambda+2\lambda_3)^2e_1e_1(\lambda)=0.
\end{align}
We claim that \begin{equation}\label{Tn38}
  \omega_{22}^1\omega_{33}^1=-\lambda_2\lambda_3-c.
\end{equation}
In fact, using the Gauss equation \eqref{Gauss-Equation} gives
$$\langle R(e_2,e_3)e_2,e_3\rangle=-\lambda_2\lambda_3-c.$$
From the definition \eqref{curvature-tensor} of curvature tensor
$R$, one can deduce that
\begin{align*}
  \langle R(e_2,e_3)e_2,e_3\rangle= & \langle \nabla_{e_2}\nabla_{e_3}e_2-\nabla_{e_3}\nabla_{e_2}e_2-\nabla_{[e_2,e_3]}e_2, e_3\rangle \\
   = & \langle \nabla_{e_2}(\omega_{32}^k e_k)-\nabla_{e_3}(\omega_{22}^k e_k)-\nabla_{(\nabla_{e_2}e_3-\nabla_{e_3}e_2)}e_2,e_3\rangle \\
   =& e_2(\omega_{32}^3)+\omega_{32}^k\omega_{2k}^3-e_3(\omega_{22}^3)-\omega_{22}^k\omega_{3k}^3
     -\omega_{23}^k\omega_{k2}^3+\omega_{32}^k\omega_{k2}^3 \\
     =& -e_2(\omega_{33}^2)+\omega_{32}^1\omega_{21}^3-e_3(\omega_{22}^3)+\omega_{22}^1\omega_{33}^1\\
    & -(\omega_{23}^1\omega_{12}^3 - \omega_{22}^3\omega_{22}^3 )
     +(\omega_{32}^1\omega_{12}^3+\omega_{33}^2\omega_{33}^2 ).
\end{align*}
In view of Lemma \ref{lambda-23} and the Codazzi equation
\eqref{Codazzi-1}, we see at once that $
\omega_{22}^3=\omega_{33}^2=0.$ In addition, it follows from (3-15)
in \cite{fu-hong2018} that $\omega_{23}^1=\omega_{32}^1=0$.
Therefore, we obtain that
\begin{equation*}
  \langle R(e_2,e_3)e_2,e_3\rangle=\omega_{22}^1\omega_{33}^1.
\end{equation*}
Consequently, our claim follows.

Putting $\omega_{ii}^1=\frac{e_1(\lambda_i)}{\lambda_i-\lambda}$
into \eqref{Tn38} gives rise to
\begin{equation}\label{Tn39}
  \frac{e_1(\lambda_2)}{\lambda_2-\lambda}\cdot \frac{e_1(\lambda_3)}{\lambda_3-\lambda}=-\lambda_2\lambda_3-c,
\end{equation}
which together with \eqref{Tn32}--\eqref{Tn33} leads to
\begin{equation}\label{Tn310}
  \big(e_1(\lambda)\big)^2
  =-\frac {(\lambda_2\lambda_3+c)(\lambda_2-\lambda)(\lambda_3-\lambda)(3\lambda+2\lambda_2)(3\lambda+2\lambda_3)}
  {9(2\lambda+\lambda_2)(2\lambda+\lambda_3)}.
\end{equation}
After eliminating $e_1e_1(\lambda)$ between \eqref{Tn36} and
\eqref{Tn37}, using \eqref{Tn310} and
$\lambda_3=-3\lambda-\lambda_2$, we get
\begin{align}\label{Tn311}
   & 3\lambda_2^6+27\lambda\lambda_2^5+(99\lambda^2-4c)\lambda_2^4 +(189\lambda^3 -24c\lambda)\lambda_2^3  \nonumber\\
   & +(196\lambda^4 -55c\lambda^2)\lambda_2^2 +(102\lambda^5-57c\lambda^3)\lambda_2  +24\lambda^6  -20c\lambda^4=0.
\end{align}
Repeating the division algorithm for polynomial to \eqref{Tn31} and
\eqref{Tn311}, we could eliminate $\lambda_2$ and obtain a
non-trivial polynomial equation of $\lambda$ with constant
coefficients as follows:
\begin{align}\label{Tn312}
   & 24\lambda^6 +(176c-28R)\lambda^4 +(196cR-18R^2-528c^2)\lambda^2\nonumber\\
   &    -3R^3+46cR^2-228c^2R +360c^3=0.
\end{align}
Therefore, we derive that $\lambda$ must be a constant, which
contradicts to our assumption. We complete the proof of Theorem
\ref{Theorem-1}.

\section{Biconservative hypersurfaces in $N^5(c)$}

In this section, we will restrict ourselves to biconservative
hypersurfaces with constant scalar curvature in $N^5(c)$. To prove
Theorem \ref{Theorem-2}, we need the following Lemma derived in
\cite{fu-hong-zhan2020}.
\begin{lemma}{\rm (\cite{fu-hong-zhan2020}) }\label{FHZ-n4}
For  $f_k=(\omega_{22}^1)^k+ (\omega_{33}^1)^k+(\omega_{44}^1)^k$,
the functions $f_1,\cdots,f_5$ satisfy:
\begin{align}
    &f_1^4-6f_1^2f_2+3f_2^2+8f_1f_3-6f_4=0, \label{f1234}\\
    &f_1^5-5f_1^3f_2+5f_1^2f_3+5f_2f_3-6f_5=0.\label{f1235}
\end{align}
\end{lemma}
By applying Lemma \ref{f1-5} and \ref{FHZ-n4}, we can prove the
following
\begin{lemma}\label{lemma-e234T}
Let $\phi: M^4\rightarrow N^5(c)$ be an orientable proper
biconservative hypersurface with constant scalar curvature. Then
$e_i(T)=0$ for $i=2,3,4$.
\end{lemma}
\begin{proof}
Substituting \eqref{L6} into \eqref{f1234} and \eqref{f1235} yields
\begin{align}
  & - T''' + 4 T T''+ 3 T'^2 -( 6 T^2  - 26 \lambda^2  + 10 c) T'
  + T^4 - (26 \lambda^2  - 10 c)T^2 \label{LM1}\\
  &+ 55 \lambda \lambda' T   - 21 \lambda \lambda'' - 12 \lambda'^2
  +27 \lambda^4 + (3R   - 72c)\lambda^2 + 9 c^2=0, \nonumber\\
  &- T''''+ (10 T' + 10 T^2 +50 \lambda^2- 10 c ) T'' - (20 T^3+20 \lambda^2 T + 20 c T \label{LM2}\\
  & - 155 \lambda \lambda' )T' + 4 T^5 + (40  c - 80 \lambda^2 )T^3 + 120 \lambda \lambda' T^2
   + (11 \lambda \lambda''  + 7 \lambda'^2    -84 \lambda^4 \nonumber\\
  &  - 48 c \lambda^2  + 36 c^2)T - 33 \lambda \lambda'''- 45 \lambda' \lambda''
   + (408 \lambda^2 + 10  R  - 252 c ) \lambda \lambda' =0.\nonumber
\end{align}
Using the above two equations, we can eliminate $T'''',~T'''$ and
obtain a second-order differential equation with respect to $T$ as
follows:
\begin{align}\label{LM3}
  &- 6 \lambda T'' + ( 18 \lambda T - 12 \lambda' )T' - 6 \lambda T^3
  + 12 \lambda' T^2 + ( 48\lambda^3  - 10 \lambda'' \\
  &  + 3 R \lambda  - 60 c\lambda  ) T + 3 \lambda''' + (27 c  -  R - 75 \lambda^2 ) \lambda'  =0. \nonumber
\end{align}
By \eqref{LM1} and \eqref{LM3}, it would allow us to eliminate
$T'''$ and $T''$. Then we get the following first-order differential
equation with respect to $T$:
\begin{align}\label{LM4}
  & a_1 T'  - a_1 T^2   + a_2 T+a_3  =0,
\end{align}
where
\begin{align*}
  a_1 =& 36 \lambda'^2 - 22 \lambda \lambda'' - 108 \lambda^4 + 3 R\lambda^2 ,\\
  a_2 =& 30 \lambda' \lambda''  - 13 \lambda \lambda'''  + 93 c\lambda \lambda'
  - 255 \lambda^3 \lambda' - 5 R\lambda \lambda',\\
  a_3=&  3 \lambda \lambda'''' - 9 \lambda' \lambda'''  + (51 \lambda^2 + 27 c -   R  )\lambda \lambda'' + (147 \lambda^2+ 3 R - 81 c ) \lambda'^2\\
  &    -162 \lambda^6  + (432  c - 18  R )\lambda^4   - 54 \lambda^2 c^2.
\end{align*}
Consider the following cases:

\textbf{Case 1.} $a_1=0$. The equation \eqref{LM4} reduces to $a_2
T+a_3  =0$.

\textbf{Case 1.1.} $a_2\neq 0$. We have $T=-a_3/a_2$, which implies
that  $e_i(T)=0$.

\textbf{Case 1.2.} $a_2=0$. Then $a_1=0$ and $a_2=0$ lead to the
following two equations:
\begin{align}
   & 36 \lambda'^2 - 22 \lambda \lambda'' - 108 \lambda^4 + 3 R\lambda^2=0, \label{CE1}\\
   & 30 \lambda' \lambda''  - 13 \lambda \lambda'''  + 93 c\lambda \lambda'
  - 255 \lambda^3 \lambda' - 5 R\lambda \lambda'=0.\label{CE2}
\end{align}
By \eqref{CE1}--\eqref{CE2}, eliminating $\lambda''',\lambda''$ and
$\lambda'$, we finally obtain a non-trivial polynomial equation of
$\lambda$:
\begin{equation*}
  96 \lambda^2 - 121 R + 1302 c=0.
\end{equation*}
Noting that both $R$ and $c$ are constant, we know that $\lambda$ is
also constant. This is a contradiction.

 \textbf{Case 2.} $a_1\neq0$.
From \eqref{LM3}--\eqref{LM4} we can elimnate the terms of $T'',~T'$
step by step and get
\begin{equation}\label{LM5}
  b_1T+b_2=0,
\end{equation}
where
\begin{align*}
   b_1= &10368 \lambda^4 \lambda'''' - 288 \lambda^2 \lambda'''' R
    + 2112 \lambda \lambda'' \lambda''''    - 3456 \lambda'^2 \lambda'''' \\
    &- 107136 \lambda^3 \lambda' \lambda''' - 972 \lambda \lambda' \lambda''' R
     + 26784 \lambda \lambda' \lambda''' c- 2730 \lambda \lambda'''^2 \\
    &    + 12540 \lambda' \lambda'' \lambda''' + 288360 \lambda^6 \lambda''
    + 9108 \lambda^4 \lambda'' R - 270864 \lambda^4 \lambda'' c \\
    & - 3336 \lambda^3 \lambda''^2  + 330516 \lambda^2 \lambda'^2 \lambda''
     - 558 \lambda^2 \lambda'' R^2 + 9108 \lambda^2 \lambda'' R c \\
    & - 7128 \lambda^2 \lambda'' c^2+ 3840 \lambda \lambda''^2 R
    - 37752 \lambda \lambda''^2 c - 2100 \lambda'^2 \lambda'' R \\
    &+ 396 \lambda'^2 \lambda'' c    - 8800 \lambda''^3
    - 467910 \lambda^5 \lambda'^2   - 42084 \lambda^3 \lambda'^2 R  \\
    & + 665604 \lambda^3 \lambda'^2 c - 244944 \lambda \lambda'^4
     + 354 \lambda \lambda'^2 R^2 - 4248 \lambda \lambda'^2 R c\\
    & - 40230 \lambda \lambda'^2 c^2 + 454896 \lambda^9
     - 4860 \lambda^7 R - 419904 \lambda^7 c   - 1188 \lambda^5 R^2\\
    &    + 31104 \lambda^5 R c - 34992 \lambda^5 c^2
     + 27 \lambda^3 R^3 - 540 \lambda^3 R^2 c + 972 \lambda^3 R c^2,\\
   b_2= &648 \lambda'^2 \lambda''''' + 54 \lambda^2 \lambda''''' R
   - 1944 \lambda^4 \lambda''''' - 396 \lambda \lambda'' \lambda''''' \\
    & + 12366 \lambda^3 \lambda' \lambda''''- 18 \lambda \lambda' \lambda'''' R
    - 1674 \lambda \lambda' \lambda'''' c+ 630 \lambda \lambda''' \lambda''''\\
    &- 1440 \lambda' \lambda'' \lambda''''  - 32076 \lambda^6 \lambda'''
    - 4158 \lambda^4 \lambda''' R  + 73224 \lambda^4 \lambda''' c \\
    & + 24066 \lambda^3 \lambda'' \lambda''' - 6876 \lambda^2 \lambda'^2 \lambda'''
     + 9 \lambda^2 \lambda''' R^2 + 486 \lambda^2 \lambda''' R c\\
     &- 11340 \lambda^2 \lambda''' c^2- 636 \lambda \lambda'' \lambda''' R
      + 2106 \lambda \lambda'' \lambda''' c+ 792 \lambda'^2 \lambda''' R \\
    & - 6156 \lambda'^2 \lambda''' c  - 1890 \lambda' \lambda'''^2
     + 2640 \lambda''^2 \lambda'''- 253098 \lambda^5 \lambda' \lambda'' \\
    & + 31554 \lambda^3 \lambda' \lambda'' R - 285876 \lambda^3 \lambda' \lambda'' c
     - 133248 \lambda^2 \lambda' \lambda''^2+ 89208 \lambda \lambda'^3 \lambda'' \\
     & + 192 \lambda \lambda' \lambda'' R^2 - 4626 \lambda \lambda' \lambda'' R c
     + 39366 \lambda \lambda' \lambda'' c^2   - 400 \lambda' \lambda''^2 R \\
     & + 10800 \lambda' \lambda''^2 c- 702756 \lambda^8 \lambda'
     + 10368 \lambda^6 \lambda' R + 506412 \lambda^6 \lambda' c \\
     & + 528174 \lambda^4 \lambda'^3  - 1863 \lambda^4 \lambda' R^2
      + 36612 \lambda^4 \lambda' R c- 323676 \lambda^4 \lambda' c^2 \\
      & - 13572 \lambda^2 \lambda'^3 R + 38988 \lambda^2 \lambda'^3 c
       - 9 \lambda^2 \lambda' R^3  + 243 \lambda^2 \lambda' R^2 c \\
      & - 3564 \lambda^2 \lambda' R c^2 + 30132 \lambda^2 \lambda' c^3
      + 29808 \lambda'^5 - 126 \lambda'^3 R^2\\
      & + 1728 \lambda'^3 R c - 1458 \lambda'^3 c^2.
\end{align*}

If $b_1=0$ and $b_2=0$, then we may eliminate the terms of
$\lambda'''''$, $\lambda''''$, $\lambda'''$, $\lambda''$, $\lambda'$
item by item, and finally obtain a non-trivial polynomial equation
concerning $\lambda$ with constant coefficients. This shows that
$\lambda$ is a constant and contradicts to our assumption. Hence the
lemma follows.
\end{proof}

According to Lemma 3.4 in \cite{fu-hong-zhan2020}, the similar proof
remains valid for the case that the number of the principal
curvatures is 2, 3 or 4. Summarizing, we have the following lemma.
\begin{lemma}\label{lemma3.4}
Let $\phi: M^4 \rightarrow N^5(c)$ be an orientable proper
biconservative hypersurface. Then $e_i(\lambda_j)=0$ for $2\leq i,
j\leq 4$.
\end{lemma}

Using Lemma \ref{lemma3.4}, the following results hold.
\begin{lemma}{\rm(c.f. \cite{fu-hong2018})}\label{lemma4.1}
If $\lambda_2$, $\lambda_3$ and $\lambda_4$ are different from each
other, then we have
\begin{align}
&\omega_{23}^4(\lambda_3-\lambda_4)=\omega_{32}^4(\lambda_2-\lambda_4)=\omega_{43}^2(\lambda_3-\lambda_2),\label{F1}\\
&\omega_{23}^4\omega_{32}^4+\omega_{34}^2\omega_{43}^2+\omega_{24}^3\omega_{42}^3=0,\label{F2}\\
&\omega_{23}^4(\omega_{33}^1-\omega_{44}^1)=\omega_{32}^4(\omega_{22}^1-\omega_{44}^1)=
\omega_{43}^2(\omega_{33}^1-\omega_{22}^1). \label{F3}
\end{align}
\end{lemma}
\begin{lemma} {\rm(c.f. \cite{fu-hong2018})}
Under the same hypotheses of Lemma $\ref{lemma4.1}$, one has
\begin{align}
\omega_{22}^1\omega_{33}^1-2\omega_{23}^4\omega_{32}^4=-\lambda_2\lambda_3-c,  \label{F4}\\
\omega_{22}^1\omega_{44}^1-2\omega_{24}^3\omega_{42}^3=-\lambda_2\lambda_4-c, \label{F5}\\
\omega_{33}^1\omega_{44}^1-2\omega_{34}^2\omega_{43}^2=-\lambda_3\lambda_4-c.\label{F6}
\end{align}
\end{lemma}
\vspace{0.2cm} \noindent{\bf The proof of Theorem
\ref{Theorem-2}:}\\

For $n=4$, combining \eqref{sum4} with \eqref{sum1} gives
\begin{equation}\label{sum5}
  \lambda_2\lambda_3+\lambda_2\lambda_4+\lambda_3\lambda_4
  =\frac 12 \Big\{\big(\sum_{i=2}^4 \lambda_i\big)^2-\sum_{i=2}^4 \lambda_i^2\Big\}=\frac 12 R+3\lambda^2-6c.
\end{equation}
Using arithmetic geometric mean inequality, it follows that
$$3\sum_{i=2}^4 \lambda_i^2\geq \big(\sum_{i=2}^4\lambda_i \big)^2,$$
which together with \eqref{sum1} and \eqref{sum4} gives
\begin{equation}\label{sum6}
  12c-R\geq 0,
\end{equation}
where the equality holds if and only if
$\lambda_2=\lambda_3=\lambda_4.$

Based on discussing the number of distinct principal curvatures, in
what follows we need to consider three cases.

\textbf{Case 1.} Assume that $M^4$ is a biconservative hypersurface
with two distinct principal curvatures in $N^5(c)$. Then
$\lambda=-2H$ and $\lambda_2 =\lambda_3 =\lambda_4=2H.$ In this
case, the scalar curvature $R$ satisfies $R=12c$. Applying a
well-known result of do Carmo and Dajczer (Theorem 4.2 of
\cite{doCarmo-Dajczer1983}), one see immediately that $M^4$ is
contained in a rotational hypersurface in $N^5(c)$.

\textbf{Case 2.} Suppose that $M^4$ is a proper biconservative
hypersurface with three distinct principal curvatures and proper
biconservative. We may assume that $\mu:=\lambda_2=\lambda_3\neq
\lambda_4$.  According to \eqref{sum1} and \eqref{sum4}, we have
\begin{align}
  2\mu+\lambda_4 =& -3\lambda ,\label{PF1}\\
  2\mu^2+\lambda_4^2 =& 12c+3\lambda^2-R.\label{PF2}
\end{align}
Eliminating $\lambda_4$ from \eqref{PF1} and \eqref{PF2}, one has
\begin{equation}\label{PF3}
  (\lambda+\mu)^2=2c-\frac 16 R.
\end{equation}
Since the scalar curvature $R$ is constant, we have
$e_1(\mu)=-e_1(\lambda).$ It follows from \eqref{BICONSERVATIVE-1}
and \eqref{BICONSERVATIVE-2} that
\begin{align}
    e_1\Big(\frac {e_1(\mu)}{\mu-\lambda}\Big)
   &=\Big(\frac {e_1(\mu)}{\mu-\lambda}\Big)^2+\lambda\mu+c,\label{PF4}\\
   e_1\Big(\frac {e_1(\lambda_4)}{\lambda_4-\lambda}\Big)
    &=\Big(\frac {e_1(\lambda_4)}{\lambda_4-\lambda}\Big)^2+\lambda\lambda_4+c.\label{PF5}
\end{align}
Using $e_1(\mu)=-e_1(\lambda)$ in \eqref{PF4}, we check at once that
\begin{equation}\label{PF6}
  (\lambda-\mu)e_1e_1(\lambda)-3\big(e_1(\lambda)\big)^2=(\lambda\mu+c)(\lambda-\mu)^2.
\end{equation}
Eliminating $\lambda_4$ between \eqref{PF1} and \eqref{PF5} and
using the fact $e_1(\mu)=-e_1(\lambda)$ again, it follows
immediately that
\begin{equation}\label{PF7}
  (4\lambda+2\mu)e_1e_1(\lambda)-3\big(e_1(\lambda)\big)^2=-(3\lambda^2+2\lambda\mu-c)(4\lambda+2\mu)^2.
\end{equation}
Noting that $\lambda,~\mu$ and $\lambda_4$ are entirely different
from each other, we conclude from \eqref{PF1} that $\lambda-\mu$ and
$4\lambda+2\mu$ can not vanish in same neighborhood. Then, we can
eliminate the terms of $e_1e_1(\lambda)$ and obtain
\begin{equation}\label{PF8}
  3(\lambda+\mu)\big(e_1(\lambda)\big)^2=- (4\lambda+2\mu)(\lambda+\mu)(\lambda-\mu)(4\lambda^2+\lambda\mu-c).
\end{equation}
Observing that $M^4$ has three distinct principal curvatures, we
deduce that $\lambda+\mu\neq 0$. Hence,
\begin{equation}\label{PF9}
 \big(e_1(\lambda)\big)^2=-\frac 13 (4\lambda+2\mu)(\lambda-\mu)(4\lambda^2+\lambda\mu-c).
\end{equation}
Similarly to \eqref{Tn38}, it follows from Gauss equation that
\begin{equation}\label{FP10}
  \omega_{22}^1\omega_{44}^1=\omega_{33}^1\omega_{44}^1=-\mu\lambda_4-c,
\end{equation}
which together with \eqref{BICONSERVATIVE-1}, \eqref{PF1} and
$e_1(\mu)=-e_1(\lambda)$ indicates that
\begin{equation}\label{PF11}
  \big(e_1(\lambda)\big)^2= (4\lambda+2\mu)(\lambda-\mu)(2\mu^2+3\lambda\mu-c).
\end{equation}
After eliminating the terms of $ \big(e_1(\lambda)\big)^2$ between
\eqref{PF9} and \eqref{PF11}, one has
\begin{equation}\label{PF12}
  2\lambda^2+5\lambda\mu+3\mu^2-2c=0.
\end{equation}
By \eqref{PF12} and \eqref{PF3}, we can eliminate $\mu^2,~\mu$ and
get a quadratic equation of $\lambda$ as follows:
\begin{equation}\label{PF13}
  2(R-12c)\lambda^2+3R^2-48Rc+192c^2=0.
\end{equation}
This clearly implies that $\lambda$ is a constant, which is  a
contradiction.

\textbf{Case 3.} Consider $M^4$ as a proper biconservative
hypersurface with four distinct principal curvatures in $N^5(c)$. We
will drive a contradiction again.  The proof will be divided into
the following two subcases.

{\bf Case 3.1.}  $\omega_{23}^4\neq0$, $\omega_{32}^4\neq0$, and
$\omega_{43}^2\neq0$. According to Lemma \ref{lemma3.4}, one has
$e_i(\omega_{ii}^1)=0$ and $e_i(\lambda_i)=0$ for $i=2, 3, 4$. Note
that \eqref{F1} and \eqref{F3} reduce to
\begin{align*}
\frac{\omega_{33}^1-\omega_{44}^1}{\lambda_3-\lambda_4}
&=\frac{\omega_{33}^1-\omega_{22}^1}{\lambda_3-\lambda_2}
=\frac{\omega_{44}^1-\omega_{22}^1}{\lambda_4-\lambda_2}:=\alpha,
\end{align*}
where $\alpha$ is a smooth function satisfying $e_i(\alpha)=0$ for
$i=2, 3, 4$. Hence there exists another smooth function $\beta$
satisfying $e_i(\beta)=0$ such that
\begin{align}
\omega_{ii}^1=\alpha \lambda_i+\beta\label{R1},\quad i=2, 3, 4.
\end{align}
Differentiating with respect to $e_1$ on both sides of equation
\eqref{R1}, using \eqref{BICONSERVATIVE-1} and
\eqref{BICONSERVATIVE-2} we get
\begin{align}
&e_1(\alpha)=\lambda(\alpha^2+1)+\alpha\beta,\label{R2}\\
&e_1(\beta)=\lambda\alpha\beta+\beta^2+c.\label{R3}
\end{align}
Taking a sum on $i$ in \eqref{R1} and using \eqref{sum1}, one has
\begin{align}
\sum_{i=2}^4\omega_{ii}^1=-3\alpha\lambda+3\beta.\label{R4}
\end{align}
Taking into account \eqref{F2} and \eqref{F4}-\eqref{F6} lead to
\begin{align*}
\omega_{22}^1\omega_{33}^1+\omega_{22}^1\omega_{44}^1+\omega_{33}^1\omega_{44}^1=
-\lambda_2\lambda_3-\lambda_2\lambda_4-\lambda_3\lambda_4-3c,
\end{align*}
 which combining with \eqref{R1} further reduces to
\begin{align}
(1+\alpha^2)(\lambda_2\lambda_3+\lambda_2\lambda_4+\lambda_3\lambda_4)+
2\alpha\beta(\lambda_2+\lambda_3+\lambda_4)+3\beta^2+3c=0.\label{R5}
\end{align}
By substituting \eqref{sum1} and \eqref{sum5} into \eqref{R5}, it
follows that
\begin{align*}
(1+\alpha^2)\times (\frac{1}{2}R+3\lambda^2-6c)-
6\alpha\beta\lambda+3\beta^2+3c=0
\end{align*}
and hence
\begin{align}\label{R6}
 6\beta^2 - 12\lambda\alpha\beta  + (6 \lambda^2 + R - 12 c) \alpha^2
 + 6\lambda^2 + R - 6c=0.
\end{align}
Moreover, differentiating \eqref{sum1} with respect to $e_1$ and
using \eqref{BICONSERVATIVE-1}, we get
\begin{align*}
-3e_1(\lambda)=\sum_{i=2}^4(\lambda_i-\lambda)\omega_{ii}^1.
\end{align*}
Using \eqref{R1}, \eqref{sum1} and \eqref{sum4}, the above equation
becomes
\begin{align}
e_1(\lambda)&=-\frac 13\sum_{i=2}^4(\lambda_i-\lambda)(\alpha
\lambda_i+\beta) \label{R7}\\
&=-\frac 13 \sum_{i=2}^4\big\{\alpha
\lambda_i^2+(\beta-\lambda\alpha)\lambda_i-\beta\lambda\big\}\nonumber\\
&=-\frac 13 \big\{\alpha(12c+3\lambda^2-R)-3\lambda(\beta-\lambda\alpha)-3\lambda\beta \big\} \nonumber\\
&=\big(\frac 13 R-4c-2\lambda^2\big)\alpha+2\lambda \beta.\nonumber
\end{align}
Differentiating \eqref{R6} with respect to $e_1$ and using
\eqref{R2}--\eqref{R3} and \eqref{R7} one has
\begin{align}
 &6 \beta^3- 18 \lambda\alpha \beta^2 + (18 \lambda^2+ 12 c - R)\alpha^2 \beta
 + 6( \lambda^2 +  c) \beta  \label{R8}\\
 & + (3  R -6 \lambda^2  - 36  c ) \lambda \alpha^3
 + (3 R - 6 \lambda^2  - 42  c ) \lambda\alpha =0.\nonumber
\end{align}
Differentiating \eqref{R8} with respect to $e_1$ and using
\eqref{R2}--\eqref{R3} and \eqref{R7}, then
\begin{align}\label{R9}
& 18 \beta^4 - 72 \alpha \beta^3 \lambda + (108  \lambda^2- 9  R
 + 108 c)\alpha^2 \beta^2 \\
 &+ (12 \lambda^2 + 24 c)\beta^2
 + (24 R- 72 \lambda^2 - 288  c) \alpha^3 \beta \lambda\nonumber\\
& + (11  R- 24  \lambda^2 - 180 c)\alpha \beta \lambda
+(18  \lambda^2 - 3   R + 36 c ) \lambda^2  \alpha^4 \nonumber\\
& + ( R^2 - 24  R c + 144  c^2)\alpha^4
+ (12 \lambda^2 + 24 c)\alpha^2 \lambda^2 \nonumber\\
& + ( R^2 - 27  R c + 180  c^2) \alpha^2    + (3 R - 6 \lambda^2  -
36  c)\lambda^2 + 6 c^2=0.\nonumber
\end{align}
Combining \eqref{R6} and \eqref{R8}, we eliminate the terms of
$\beta^3,~\beta^2$ and derive
\begin{align}\label{S1}
  (R - 12 c) \big\{ (2 \alpha^2  + 1)\beta -4 \alpha^3 \lambda  - 4 \alpha \lambda \big\}=0.
\end{align}
From \eqref{R6} and \eqref{S1} we can continue to eliminate
$\beta^2,~\beta$ and obtain
\begin{align}\label{S2}
   &  6(\alpha^2 + 1)(4\alpha^4 + 12\alpha^2 + 1)\lambda^2 \\
   & +(  4 \alpha^6  + 8\alpha^4 + 5\alpha^2 + 1 )R\nonumber\\
   &- (  48\alpha^6 + 72\alpha^4 + 36\alpha^2 + 6 )c =0.\nonumber
\end{align}
On the other hand, we can eliminate the terms of
$\beta^4,~\beta^3,~\beta^2$ by \eqref{R6} and \eqref{R9}, then we
have
\begin{align}\label{S3}
  &(R - 12 c) \big\{ ( 12 \alpha^2 + 22) \alpha \beta \lambda
   +( 12 \lambda^2 + 6 R - 72 c) \alpha^4 \\
  & + (26 \lambda^2 + 7 R - 68 c )\alpha^2
  + (14 \lambda^2 + R - 8 c) \big\}=0.\nonumber
\end{align}
Eliminating $\beta$ by \eqref{S1} and \eqref{S3}, we get
\begin{align}\label{S4}
   & 2(\alpha^2 + 1)(36\alpha^4 + 64\alpha^2 + 7)\lambda^2 \\
   & +( 12 \alpha^6 + 20 \alpha^4 + 9 \alpha^2 + 1)R     \nonumber  \\
   &  - ( 144 \alpha^6  + 208 \alpha^4 + 84 \alpha^2 + 8 )c=0.\nonumber
\end{align}
Similarly, eliminating $\lambda^2$ between \eqref{S2} and
\eqref{S4}, we finally obtain
\begin{align}\label{S5}
  (32  R - 504  c)\alpha^6 + (20  R - 420 c) \alpha^4 - (14  R - 30  c)\alpha^2 - (2 R - 9 c)=0.
\end{align}
 We conclude from \eqref{S5} that $\alpha$ must be a constant, which together with \eqref{S2} implies that $\lambda$ is a constant. This is a contradiction.

{\bf Case 3.2.} $\omega_{23}^4=\omega_{32}^4=\omega_{43}^2=0$. In
this case, it follows  from \eqref{F4}, \eqref{F5}, \eqref{F6} that
\begin{align}
\omega_{22}^1\omega_{33}^1=-\lambda_2\lambda_3-c,  \label{G1}\\
\omega_{22}^1\omega_{44}^1=-\lambda_2\lambda_4-c, \label{G2}\\
\omega_{33}^1\omega_{44}^1=-\lambda_3\lambda_4-c.\label{G3}
\end{align}
Taking the sum of \eqref{G1}, \eqref{G2} and \eqref{G3} gives
\begin{eqnarray}\label{G4}
\omega_{22}^1\omega_{33}^1+\omega_{22}^1\omega_{44}^1+\omega_{33}^1\omega_{44}^1
=-3c-\lambda_2\lambda_3-\lambda_2\lambda_4-\lambda_3\lambda_4.
\end{eqnarray}
Substituting \eqref{sum5} into \eqref{G4} gives
\begin{align}\label{G6}
\omega_{22}^1\omega_{33}^1+\omega_{22}^1\omega_{44}^1+\omega_{33}^1\omega_{44}^1
=3c-3\lambda^2-\frac 12 R.
\end{align}
From the expressions of $f_1$ and $f_2$, we have
\begin{align}\label{G7}
\omega_{22}^1\omega_{33}^1+\omega_{22}^1\omega_{44}^1+\omega_{33}^1\omega_{44}^1
=\frac 12 (f_1^2-f_2)=\frac 12 (T^2-T'-3\lambda^2+3c).
\end{align}
Combining \eqref{G6} with \eqref{G7} gives
\begin{equation}\label{G8}
 T^2 - T'+ 3\lambda^2-3c+R=0.
\end{equation}
Using \eqref{G1}--\eqref{G3} again, it follows that
\begin{equation}\label{G9}
(\omega_{22}^1\omega_{33}^1\omega_{44}^1)^2+(\lambda_2\lambda_3+c)(\lambda_2\lambda_4+c)(\lambda_3\lambda_4+c)=0.
\end{equation}
Let $K=\lambda_2\lambda_3\lambda_4$. By \eqref{sum1} and
\eqref{sum5}, we obtain
\begin{align}\label{G10}
&(\lambda_2\lambda_3+c)(\lambda_2\lambda_4+c)(\lambda_3\lambda_4+c)\\
=&(\lambda_2\lambda_3\lambda_4)^2+c^2(\lambda_2\lambda_3+\lambda_2\lambda_4+\lambda_2\lambda_4)
+c\lambda_2\lambda_3\lambda_4(\lambda_2+\lambda_3+\lambda_4)+c^3\nonumber\\
=&K^2+c^2\Big(\frac 12 R+3\lambda^2-6c\Big)-3c\lambda K+c^3\nonumber\\
=& K^2-3c\lambda K+3c^2\lambda^2 -5c^3+\frac 12 c^2 R.\nonumber
\end{align}
An easy computation shows that
\begin{align*}
f_1^3-f_3=&3\big\{(\omega_{22}^1)^2(\omega_{33}^1+\omega_{44}^1)
+(\omega_{33}^1)^2(\omega_{22}^1+\omega_{44}^1)\\
& +(\omega_{44}^1)^2(\omega_{22}^1+\omega_{33}^1)\big\}+6\omega_{22}^1\omega_{33}^1\omega_{44}^1\\
=&3\sum_{i=2}^4(\omega_{ii}^1)^2(f_1-\omega_{ii}^1)+6\omega_{22}^1\omega_{33}^1\omega_{44}^1\\
=&3f_1f_2-3f_3+6\omega_{22}^1\omega_{33}^1\omega_{44}^1,
\end{align*}
which together with the expressions of $f_1$, $f_2$ and $f_3$
implies that
\begin{align}\label{G11}
\omega_{22}^1\omega_{33}^1\omega_{44}^1=&\frac 16(f_1^3-3f_1f_2+2f_3)\\
=&\frac
16(T^3-11\lambda^2T+7cT+T''-3TT'+12\lambda\lambda').\nonumber
\end{align}
Substituting \eqref{G10} and \eqref{G11} back to \eqref{G9} yields
to
\begin{align}\label{G12}
&(T^3-11\lambda^2T+7cT+T''-3TT'+12\lambda\lambda')^2\\
&+36\Big(K^2-3c\lambda K+3c^2\lambda^2 -5c^3+\frac 12 c^2
R\Big)=0.\nonumber
\end{align}
Using \eqref{sum1} and \eqref{sum2} again, we find
\begin{align}\label{G13}
\lambda_3\lambda_4=&\frac 12 \Big\{(\lambda_3+\lambda_4)^2-(\lambda_3^2+\lambda_4^2)\Big\}\\
=&\frac 12 \Big\{(-3\lambda-\lambda_2)^2-(12c+3\lambda^2-R-\lambda_2^2) \Big\} \nonumber\\
=&3\lambda^2+3\lambda\lambda_2+\lambda_2^2-6c+\frac 12 R.\nonumber
\end{align}
Similarly, we can see that
\begin{align}
\lambda_2\lambda_4=&3\lambda^2+3\lambda\lambda_3+\lambda_3^2-6c+\frac 12 R,\label{G14}\\
\lambda_2\lambda_3=&3\lambda^2+3\lambda\lambda_4+\lambda_4^2-6c+\frac
12 R.\label{G15}
\end{align}
Applying \eqref{G13}--\eqref{G15} yields
\begin{align}\label{G16}
&\omega_{22}^1\lambda_3\lambda_4+\omega_{33}^1\lambda_2\lambda_4+\omega_{44}^1\lambda_2\lambda_3\\
=& (3\lambda^2-6c+\frac 12
R)(\omega_{22}^1+\omega_{33}^1+\omega_{44}^1)
+3\lambda(\lambda_2 \omega_{22}^1+\lambda_3 \omega_{33}^1+\lambda_4 \omega_{44}^1)\nonumber\\
&+(\lambda_2^2 \omega_{22}^1+\lambda_3^2 \omega_{33}^1+\lambda_4^2 \omega_{44}^1)\nonumber\\
=&(3\lambda^2-6c+\frac 12 R)f_1+3\lambda g_1+g_3.\nonumber
\end{align}
which together with the expression of $f_1$, $g_1$ and $g_3$ leads
to
\begin{align}\label{G17}
K':=&e_1(K)=e_1(\lambda_2)\lambda_3\lambda_4+e_1(\lambda_3)\lambda_2\lambda_4+e_1(\lambda_4)\lambda_2\lambda_3\\
=&\omega_{22}^1(\lambda_2-\lambda)\lambda_3\lambda_4
+\omega_{33}^1(\lambda_3-\lambda)\lambda_2\lambda_4
+\omega_{44}^1(\lambda_4-\lambda)\lambda_2\lambda_3\nonumber\\
=&\lambda_2\lambda_3\lambda_4(\omega_{22}^1+\omega_{33}^1+\omega_{44}^1)
-\lambda(\omega_{22}^1\lambda_3\lambda_4+\omega_{33}^1\lambda_2\lambda_4+\omega_{44}^1\lambda_2\lambda_3)\nonumber\\
=&Kf_1-  \lambda(3\lambda^2-6c+\frac 12 R)f_1 - 3\lambda^2 g_1 - \lambda g_3\nonumber\\
=&KT-7\lambda^3 T + 6 c\lambda T -\frac 12 \lambda R T
+9\lambda^2\lambda'.\nonumber
\end{align}
Differentiating \eqref{G12} with respect to $e_1$ and using
\eqref{G17} give rise to
\begin{align}\label{G18}
 &36 K^2 T + (162 T \lambda c + 324 \lambda^2 \lambda' - 18 R T \lambda - 252 T \lambda^3 - 54 \lambda' c)K \\
&+ 27 R T \lambda^2 c + 3 T^5  T' - 3 T^4  T'' - 22 T^4 \lambda \lambda' - 12 T^3  T'^2 - 44 T^3  T' \lambda^2 \nonumber\\
&+ 28 T^3  T' c + T^3  T''' + 12 T^3 \lambda \lambda'' + 12 T^3 \lambda'^2 + 12 T^2  T'  T'' \nonumber\\
&+ 102 T^2  T' \lambda \lambda' + 33 T^2  T'' \lambda^2 - 21 T^2  T'' c + 242 T^2 \lambda^3 \lambda' - 154 T^2 \lambda \lambda' c \nonumber\\
&+ 9 T  T'^3 + 66 T  T'^2 \lambda^2 - 42 T  T'^2 c - 3 T  T'  T''' + 121 T  T' \lambda^4\nonumber\\
& - 154 T  T' \lambda^2 c - 36 T  T' \lambda \lambda'' - 36 T  T' \lambda'^2 + 49 T  T' c^2 - 3 T  T''^2\nonumber\\
& - 58 T  T'' \lambda \lambda'- 11 T  T''' \lambda^2 + 7 T  T''' c + 378 T \lambda^4 c - 132 T \lambda^3 \lambda'' \nonumber\\
&- 396 T \lambda^2 \lambda'^2 - 324 T \lambda^2 c^2 + 84 T \lambda \lambda'' c + 84 T \lambda'^2 c - 3  T'^2  T'' \nonumber\\
&- 36  T'^2 \lambda \lambda' - 11  T'  T'' \lambda^2 + 7  T'  T'' c - 132  T' \lambda^3 \lambda' + 84  T' \lambda \lambda' c\nonumber\\
& +  T''  T''' + 12  T'' \lambda \lambda'' + 12  T'' \lambda'^2 + 12  T''' \lambda \lambda' - 486 \lambda^3 \lambda' c \nonumber\\
&+ 144 \lambda^2 \lambda' \lambda'' + 144 \lambda \lambda'^3 + 108 \lambda \lambda' c^2=0.\nonumber\\
\end{align}
Eliminating the terms of $K^2$ from \eqref{G12} and \eqref{G18}
shows that
\begin{equation}\label{G20}
  d_1K+d_2=0,
\end{equation}
where
\begin{align*}
  d_1= & 18 R T \lambda + 252 T \lambda^3 - 270 T \lambda c - 324 \lambda^2 \lambda' + 54 \lambda' c , \\
  d_2= & - 27 R T \lambda^2 c + 18 R T c^2 + T^7 - 9 T^5 T' - 22 T^5 \lambda^2 + 14 T^5 c \\
     & + 5 T^4 T'' + 46 T^4 \lambda \lambda' + 21 T^3 T'^2 + 110 T^3 T' \lambda^2 - 70 T^3 T' c \\
    &- T^3 T''' + 121 T^3 \lambda^4 - 154 T^3 \lambda^2 c - 12 T^3 \lambda \lambda'' - 12 T^3 \lambda'^2 \\
    &+ 49 T^3 c^2 - 18 T^2 T' T'' - 174 T^2 T' \lambda \lambda' - 55 T^2 T'' \lambda^2 + 35 T^2 T'' c \\
    &- 506 T^2 \lambda^3 \lambda' + 322 T^2 \lambda \lambda' c - 9 T T'^3 - 66 T T'^2 \lambda^2 + 42 T T'^2 c \\
    &+ 3 T T' T''' - 121 T T' \lambda^4 + 154 T T' \lambda^2 c + 36 T T' \lambda \lambda'' + 36 T T' \lambda'^2 \\
&- 49 T T' c^2 + 4 T T''^2 + 82 T T'' \lambda \lambda' + 11 T T''' \lambda^2 - 7 T T''' c \\
&- 378 T \lambda^4 c + 132 T \lambda^3 \lambda'' + 540 T \lambda^2 \lambda'^2 + 432 T \lambda^2 c^2 - 84 T \lambda \lambda'' c \\
&- 84 T \lambda'^2 c - 180 T c^3 + 3 T'^2 T'' + 36 T'^2 \lambda \lambda' + 11 T' T'' \lambda^2 - 7 T' T'' c \\
&+ 132 T' \lambda^3 \lambda' - 84 T' \lambda \lambda' c - T'' T''' - 12 T'' \lambda \lambda'' - 12 T'' \lambda'^2 \\
&- 12 T''' \lambda \lambda' + 486 \lambda^3 \lambda' c - 144
\lambda^2 \lambda' \lambda'' - 144 \lambda \lambda'^3 - 108 \lambda
\lambda' c^2
\end{align*}

When $d_1=0$, it follows immediately that $d_2=0$. As in the proof
of Lemma \ref{lemma-e234T}, we can certainly eliminate $T'''$,
$T''$, $T'$ and $T$ gradually, and obtain a polynomial equation
concerning $\lambda$ and its derivatives. In addition, similar
arguments applied to \eqref{G8} and $d_1=0$ gives another polynomial
equation concerning $\lambda$ and its derivatives. From these two
equations, we may eliminate all the derivatives of $\lambda$, and
get a non-trivial polynomial equation of $\lambda$ with constant
coefficients. Therefore, the mean curvature $H$ must be constant, a
contradiction.

Now we assume that $d_1\neq 0$. By use of \eqref{G20}, we may
eliminate $K$ in \eqref{G12} and have
\begin{align}\label{G21}
&(T^3-11\lambda^2T+7cT+T''-3TT'+12\lambda\lambda')^2\\
&+36\Big((\frac {d_2}{d_1})^2+3c\lambda \frac
{d_2}{d_1}+3c^2\lambda^2 -5c^3+\frac 12 c^2 R\Big)=0.\nonumber
\end{align}
Using an analysis similar to the above, we can eliminate $T'''$,
$T''$, $T'$ and $T$ between \eqref{G8} and \eqref{G21}, and deduce a
polynomial equation concerning $\lambda$ and its derivatives. In the
proof of Lemma \ref{lemma-e234T}, we find that if $a_1=0$, then
\eqref{LM4} shows that $a_2 T+a_3=0$, which together with \eqref{G8}
can deduce another polynomial equation concerning $\lambda$ and its
derivatives after eliminating the terms of $T'$ and $T$.  The rest
of the proof runs as before. If $a_1\neq 0$, the three equations
\eqref{LM4}, \eqref{G8} and \eqref{G21} can be handled in the same
way. We complete the proof of Theorem \ref{Theorem-2}.

\section{A final remark}

According to Theorem \ref{Theorem-2}, any non-CMC biconservative
hypersurface with constant scalar curvature in the 5-dimensional
space form $N^5(c)$ is a certain rotational hypersurface with two
principal curvatures verifying
\begin{eqnarray}\label{FU1}
-\lambda_1=\lambda_2=\lambda_3=\lambda_4.
\end{eqnarray}
A natural question is whether the converse is also true. In the
following, we show that any non-CMC rotational hypersurface
verifying \eqref{FU1} must be biconservative.

Let us recall the explicit parametric equation of the rotational
hypersurface in $\mathbb S^5$ from \cite{doCarmo-Dajczer1983}.
Consider the profile curve $\gamma$  by
\begin{eqnarray}
\gamma(s)=(h_1(s), 0, 0, h_2(s), h_3(s))\nonumber
\end{eqnarray}
for some smooth function $h_i$ $(i=1, 2, 3)$.  Then the
parametrization of the rotational hypersurface in $\mathbb S^5$ can
be written as
\begin{align}
f(s, t_1, t_2, t_{3})=\big(h_1(s)\varphi_1, h_1(s)\varphi_2,
h_1(s)\varphi_3, h_2(s), h_3(s)\big),\nonumber
\end{align}
where $\varphi(t_1, t_2, t_3)=(\varphi_1, \varphi_2, \varphi_3)$ is
an orthogonal parametrization of the unit sphere. Hence
$\varphi_i=\varphi_i(t_1, t_2, t_3)$ and $
\varphi_1^2+\varphi_2^2+\varphi_3^2=1$. Since the profile curve
$\gamma$ belongs to $\mathbb S^5$ and the parameter $s$ can be
chosen as its arc length, it follows that
\begin{align}
h_1^2+h_2^2+h_3^2=1\,\,{\rm and
}\,\,\,h_1'^2+h_2'^2+h_3'^2=1.\nonumber
\end{align}
It is straightforward to compute that
\begin{align}
&\frac{\partial f}{\partial s}=\big(h'_1\varphi_1,
h_1'\varphi_2, h_1'\varphi_3, h'_2, h'_3\big),\nonumber\\
&\frac{\partial f}{\partial t_i}=\big(h_1\frac{\partial
\varphi_1}{\partial t_i}, h_1\frac{\partial \varphi_2}{\partial
t_i}, h_1\frac{\partial \varphi_3}{\partial t_i}, 0, 0\big), \quad
i=2, 3, 4.\nonumber
\end{align}
Choosing a frame $\{\frac{\partial f}{\partial s}, \frac{\partial
f}{\partial t_1}, \frac{\partial f}{\partial t_2}, \frac{\partial
f}{\partial t_3}\}$, do Carmo-Dajczer \cite{doCarmo-Dajczer1983}
further showed that the principal curvatures $\lambda_i$ of $M^4$
are given by
\begin{align}
&\lambda_1=\frac{h_1''+h_1}{\sqrt{1-h_1^2-h_1'^2}},\label{FU2}\\
&\lambda_i=-\frac{\sqrt{1-h_1^2-h_1'^2}}{h_1},\quad i=2, 3,
4.\label{FU3}
\end{align}
Notice that the principal curvatures $\lambda_i$ are functions
depending only on the variable $s$.

Since the principal curvatures satisfy \eqref{FU1} on the rotational
hypesurfaces in Theorem \ref{Theorem-2}, from \eqref{FU2} and
\eqref{FU3} we get a second order ODE that
\begin{equation}
h_1(h_1''+h_1)=1-h_1^2-h_1'^2 \nonumber
\end{equation}
and the mean curvature $H$ is given by
\begin{equation}
H=-\frac{h_1''+h_1}{2\sqrt{1-h_1^2-h_1'^2}}. \nonumber
\end{equation}

It is straightforward to check that $\frac{\partial }{\partial s}$
is a principal direction and the corresponding principal curvature
is $\lambda_1 \,(=-2H)$. Hence we conclude that this rotation
hypersurface $M^4$ satisfies \eqref{biharmonic condition} and must
be biconservative. Similar arguments can be applied to the cases for
the rotational hypersurfaces with \eqref{FU1} in $N^5(c)$ for
$c\leq0$.

\medskip

\medskip\noindent
{\bf Acknowledgement:} {The authors are supported by Liaoning
Provincial Science and Technology Department Project (No.
2020-MS-340), the NSFC (No. 11801246, 12101083) and Natural Science
Foundation of Jiangsu Province (No. BK20210936)}



\end{document}